\documentclass[11pt]{amsart}

\usepackage{verbatim}

\renewcommand{\AA}{\mathcal{A}}
\newcommand{\RX}{\mathcal{R}(X)}

\newcommand{\ZZ}{\mathbb{Z}}
\newcommand{\groupid}{e}
\newcommand{\emptyword}{\lambda}

\newtheorem{theorem}{Theorem}
\newtheorem{corollary}[theorem]{Corollary}
\newtheorem{lemma}[theorem]{Lemma}

\theoremstyle{definition}
\newtheorem{remark}[theorem]{Remark}

\begin{document}

\title{No positive cone in a free product is regular}
\author{Susan Hermiller}
\address{Department of Mathematics, University of Nebraska, Lincoln, NE 68588-0130, USA}
\email{hermiller@unl.edu}
\author{Zoran \v{S}uni\'c}
\address{Department of Mathematics, Hofstra University, Hempstead, NY 11549, USA}
\email{zoran.sunic@hofstra.edu}

\begin{abstract}
We show that there exists no left order on the free product of two nontrivial, finitely generated, left-orderable groups such that the corresponding positive cone is represented by a regular language. Since there are orders on free groups of rank at least two with positive cone languages that are context-free (in fact, 1-counter languages), our result provides a bound on the language complexity of positive cones in free products that is the best possible within the Chomsky hierarchy. It also provides a strengthening of a result by Crist{\'o}bal Rivas which states that the positive cone in a free product of nontrivial, finitely generated, left-orderable groups cannot be finitely generated as a semigroup. As another illustration of our method, we show that the language of all geodesics (with respect to the natural generating set) that represent positive elements in a graph product of groups defined by a graph of diameter at least 3 cannot be regular. 
\end{abstract}



\maketitle

%---------------------------------------

\section{Introduction}

\subsection{Basic definitions, background, and notation}
A  total left \emph{order} on a group $G$ is a total order  $\leq$ on $G$ compatible with left multiplication; that is, $g \leq g'$ implies $hg \leq hg'$ for all $g,g',h \in G$.  Throughout this paper we assume that all group orders are total left orders (and we will not use the adjectives total or left again).

A group $G$ is called \emph{orderable} if $G$ admits an order.  Saying that $(G,\leq)$ is an \emph{ordered group} means that we consider the group $G$ along with some specific order $\leq $ on $G$.

The \emph{positive cone} of an ordered group $(G,\leq)$ is the set $G_+=\{g \in G\mid e < g\}$ of positive elements with respect to $\leq$. The positive cone is a subsemigroup of $G$ and $G$ is 
partitioned as $G = G_+ \sqcup \{\groupid\} \sqcup G_-$, where $\groupid$ is the identity element of $G$ and $G_-=(G_+)^{-1}$ is the set of negative elements in $G$.

Let $X$ be a finite set of symbols and $\phi: X^* \to G$ a surjective homomorphism from the free monoid $X^*$ to the group $G$.   We sometimes suppress $\phi$ from the notation and denote $\phi(w)$ by 
$\overline{w}$ and $\phi(L)$ by $\overline{L}$.

A \emph{language} over $X$ is a subset $L \subset X^*$.  Let $\emptyword$ denote the empty word. The class of \emph{regular} languages over $X$ is the closure of the finite languages under the 
operations of finite union, finite intersection, complementation, concatenation, and Kleene star.  Regular languages are the languages accepted by finite state automata.  The reverse of any regular 
language is regular, and for any finite set $Y$ and monoid homomorphism $\alpha:X^* \rightarrow Y^*$ the image $\alpha(L)$ of any regular language $L$ over $X$ is a regular language over $Y$.

We say that a language $L \subseteq X^*$ \emph{represents} a subset $S$ of $G$ if $\phi(L)=S$. A subset $S$ of $G$ represented by a regular language is often called a rational subset of $G$. A language over $X$ that represents the positive cone $G_+$ of an ordered group $(G,\leq)$ is called a 
\emph{positive cone language} for $(G,\leq)$ over $X$ (and with respect to $\phi$).

%---------------------------------------

\subsection{Regular positive cones}

We are interested in the question of which finitely generated groups admit a representation of a positive cone by a regular language. Several examples of such groups have been shown to exist.

As a first example, note that there are orders on $\ZZ^2$ that have a regular positive cone language. Indeed, let $X=\langle x,y,x^{-1},y^{-1}\rangle$. The regular language
\[
 \{x^n y^m \mid n \geq 1,~m \in \ZZ\} \cup \{y^m \mid m \geq 1\}
\]
represents the positive cone of the ``lexicographic order'' on $\ZZ^2 = \langle \overline{x},\overline{y}\rangle$.

Both the Dehornoy order~\cite{dehornoy:order-braid} and the Dubrovina-Dubrovin order~\cite{dubrovina-d:order-braid} on the braid group $B_n$, with $n \geq 3$, admit positive cone languages that are regular.  Moreover, the positive cone for the latter is finitely generated as a semigroup.

Rourke and Wiest~\cite{rourke-w:automatic} provided regular positive cone languages for certain orders on mapping class groups of compact surfaces with a finite number of punctures and non-empty boundary.

On the negative side, Calegari~\cite{calegari:problems3} showed that no hyperbolic 3-manifold group admits a positive cone language that is simultaneously regular and geodesic. His argument applies only to languages consisting of (quasi-)geodesics because it uses the fact that hyperbolic 3-manifold groups contain quasi-geodesically embedded copies of the free group of rank 2. Note that the issue is subtle, since groups admitting regular positive cone languages may contain free subgroups. Indeed, braid groups along with the Dehornoy orders and the Dubrovina-Dubrovin orders provide such examples. 

%---------------------------------------

\subsection{The main result}

It is well known that the free product of two orderable groups is orderable~\cite{vinogradov:ordered}. Rivas~\cite{rivas:free-product} showed that the space of all orders on a free product $G=A*B$, 
where $A$ and $B$ are nontrivial, finitely generated, orderable groups, is uncountable and has the structure of a Cantor set. Navas~\cite{navas:ordered-dynamics} pointed out that, in the case of the 
free group $F_k$ of finite rank $k$, $k \geq2$, the same conclusion follows from the earlier work of McCleary~\cite{mccleary:free-lattice-ordered}, and also provided another proof in this case. In our main result 
we show that, despite such an abundance of available orders, a positive cone language in a free product is never regular.

\begin{theorem}\label{t:not-regular-free}
Let $A$ and $B$ be two nontrivial, finitely generated, orderable groups. There exists no order on $G=A*B$ such that its positive cone is represented by a regular language (for any finite alphabet $X$, 
any homomorphism $\phi:X^* \to G$, and any choice of a language representing the positive cone).
\end{theorem}

The orders defined on free groups $F_k$ in~\cite{sunic:free-lex,sunic:from-oriented} have context-free positive cone languages. More specifically, the sets of freely reduced words representing the 
elements in their positive cones are 1-counter languages (the stack of the push-down automaton uses a one-letter alphabet). In light of Theorem~\ref{t:not-regular-free}, it follows that the orders on 
$F_k$ from~\cite{sunic:free-lex,sunic:from-oriented} are, at least from the language theoretic point of view, among the simplest ones possible in the context of free products.

%---------------------------------------

\subsection{Other results and remarks}

The main result immediately provides the following corollary.

\begin{corollary}\label{c:nfg}
Let $A$ and $B$ be two nontrivial, finitely generated, orderable groups. There exists no order on $G=A*B$ such that its positive cone is finitely generated as a semigroup. 
\end{corollary}

This corollary was already established by Rivas~\cite{rivas:free-product} by using a very different approach.  In particular, 
Linnell~\cite{linnell:retracted},~\cite[Proposition~1.8]{navas:ordered-dynamics} observed that if a positive cone on a finitely generated group $G$ is a finitely generated semigroup, then the 
corresponding order is isolated in the space of all orders. Since the space of orders on $G=A*B$ has no isolated points~\cite{rivas:free-product}, no order on $G$ has a finitely generated positive 
cone. This is an excellent approach, but it seems that, in general, it is not easy to establish that there are no isolated points in the space of orders of some group (including the case of the free 
group). Also, the ``no isolated points'' approach is not helpful in establishing stronger results in the spirit of Theorem~\ref{t:not-regular-free}. Namely, the space of orders on $\ZZ^2$ is a Cantor set, but, as we already mentioned, there are orders on $\ZZ^2$ that have a regular positive cone language. 

Another approach to Corollary~\ref{c:nfg} is provided by Kielak~\cite{kielak:ends} who proved that a finitely generated group with infinitely many ends cannot be a fraction group of a proper, finitely generated subsemigroup. 

The proof of Theorem~\ref{t:not-regular-free} is based on the following lemma.

\begin{lemma}\label{l:not-regular}
Let $(G,\leq)$ be a finitely generated ordered group and let $L \subseteq X^*$ be a language representing the positive cone $G_+$. Denote by $Pref(L)$ the language of prefixes of the words in $L$. If 
there exists a set of positive elements $T$ in $G$ such that $T$ is unbounded (with respect to $\leq$) from above and
\[
 T^{-1} \subseteq \overline{Pref(L)},
\]
then $L$ is not regular.
\end{lemma}

In our next result we consider graph products of groups which contain subgroups that are free products of pairs of vertex subgroups.  (In the following we use the convention that edges of the defining graph represent commuting relations). The graph product construction preserves many properties of groups; for example, Chiswell~\cite{chiswell:ordering-graph-products} has shown that graph products preserve orderability, and Loeffler, Meier and Worthington~\cite{loeffler-m-w:graph-products} have shown that the graph product preserves regularity of the set of all geodesics, using the union of the vertex group  generating sets.  One important family of graph products is the class of right-angled Artin groups, which are graph products of cyclic groups of infinite order; in their case both the language of geodesics and the language of all geodesics representing elements in the positive cone are regular for each vertex group. In order to illustrate that the applicability of Lemma~\ref{l:not-regular} goes beyond free products, we modify the proof of Theorem~\ref{t:not-regular-free} to show that regularity of the geodesic positive cone languages of orders on the vertex groups cannot be preserved in the graph product for graphs of diameter at least 3, when using the same union of the vertex group generating sets.

\begin{theorem}\label{t:not-regular-ag}
Let $\Gamma$ be a finite, simple graph of diameter at least 3 and $V$ its vertex set.  For each $v \in V$ let $G_v$ be a nontrivial, orderable group with generating set $\overline{Y_v}$ for a finite 
set $Y_v$, and let $X = \sqcup_{v \in V} (Y_v \sqcup Y_v^{-1})$.
Let $G_\Gamma$ be the associated graph product group, and suppose that $\phi: X^* \to G_\Gamma$ has the property that $\phi(x^{-1})=\phi(x)^{-1}$ for all $x \in X$.   Let $Geo$ be the language of all 
geodesics for $G_\Gamma$ with respect to $X$. There exists no order $\leq$ on $G_\Gamma$ such that the positive cone language
\[
 Geo_+ = \{ \ w \in Geo \mid w \textup{ represents a positive element in } G_\Gamma \ \}
\]
is regular.
\end{theorem}

\begin{remark}
Let $X = Y \sqcup Y^{-1}$, for some finite set $Y$, and let $\phi:X^* \to G$ be a homomorphism such that $\phi(x^{-1})=\phi(x)^{-1}$ for all $x \in X$. Let $L_+ \subseteq X^*$ be a regular language 
representing the positive cone in the ordered group $(G,\leq)$. Since the homomorphic image of the reverse of a regular language is regular, the language $(L_+)^{-1}$ of formal inverses of the words 
in $L_+$ is also regular, and hence the union $L = L_+ \sqcup \{\emptyword\} \sqcup (L_+)^{-1}$
%, where $\emptyword$ is the empty word and , 
is a regular language representing the entire group $G$ (i.e., $\phi(L)=G$).
Thus, every regular language $L_+$ over $X$ representing the positive cone $G_+$ is induced from some regular language $L$ representing $G$, by
\[
 L_+ = \{ \ w \in L \mid w \textup{ represents a positive element in } G \ \}.
\]

Using this viewpoint, Theorem~\ref{t:not-regular-ag} implies, in particular, that no regular positive cone language can be induced from the regular language of all geodesics in a right-angled Artin 
group on a graph $\Gamma$, when the diameter of the graph $\Gamma$ is at least 3.
\end{remark}

%---------------------------------------

\section{Proofs}

We first observe that the question of representability of a subset of a group by a regular language is independent of the choice of the alphabet $X$ and the homomorphism $\phi:X^* \to G$.

\begin{lemma}\label{l:independence}
Let $X$ and $Y$ be two finite alphabets and $\phi_X:X^* \to G$ and $\phi_Y:Y^* \to G$ be two surjective homomorphisms to the finitely generated group $G$. Let $S$ be a subset of $G$ represented by a 
regular language $L \subseteq X^*$. Then there exists a regular language $L'\subseteq Y^*$ representing $S$.
\end{lemma}

\begin{proof}
For every letter $x \in X$, choose a word $w_x$ over $Y$ such that $\phi_Y(w_x) = \phi_X(x)$.  Define a monoid homomorphism $\alpha:X^* \rightarrow Y^*$ by $\alpha(x)=w_x$ for all $x \in X$.
The homomorphic image $\alpha(L)$ is a regular language over $Y$ representing $S$.
\end{proof}

Because of the independence on the alphabet, we can always work with a group alphabet $X = Y \sqcup Y^{-1}$, for some finite set $Y$, and $\phi:X^* \to G$ with the property that 
$\phi(x^{-1})=\phi(x)^{-1}$ for all $x \in X$. Moreover, by Benois' Lemma~\cite{benois:free} (see~\cite{bartholdi-s:rational} for an exposition), we can always make a further simplification and assume that a regular language representing a set of elements of $G$ lies within the set $\RX$ of freely reduced group words over $X$.  

\begin{lemma}[Benois' Lemma]\label{l:reduced}
Let $X = Y \sqcup Y^{-1}$ be a group alphabet, for some finite set $Y$, and let $\phi:X^* \to G$ be a surjective homomorphism to a group $G$ such that $\phi(x^{-1})=\phi(x)^{-1}$ for all $x \in X$. 
Let $L$ be a regular language over $X$ that represents $S$ in $G$. Then there exists a regular language $L' \subseteq \RX$ representing $S$. Moreover, one such language is the language of freely reduced 
words of the words in $L$. 
\end{lemma}

\begin{proof}[Proof of Lemma~\ref{l:not-regular}]
By way of contradiction, assume that $L$ is regular and let $\AA$ be an automaton on $k$ states accepting $L$.

Consider the ball $B$ of radius $k-1$ in $G$ with respect to $\overline{X}$. Since $T$ is unbounded from above, there exists an element $t \in T$ that is greater (with respect to the order $\leq$) 
than every element in $B$. Then $t^{-1} \in T^{-1}$, and so there exists a word $w$ in $Pref(L)$ representing $t^{-1}$. Now $w \in Pref(L)$ implies that there exists a word $u$ such that $wu \in L$. 
Since the word $wu$ is accepted by the automaton $\AA$ on $k$ states, there must be a word $v$ of length at most $k-1$, such that the word $wv$ is also accepted by the automaton. Therefore 
$\groupid<\overline{wv} = t^{-1}\overline{v}$, which implies that $t < \overline{v}$. This is impossible since $\overline{v}$ is an element in $B$.
\end{proof}

Our proof of Theorem~\ref{t:not-regular-free} also relies on the structure of words representing elements in a free product.  Let $A$ and $B$ be nontrivial groups. For each $g \in A*B$, the 
\emph{reduced factorization} of $g$ is the unique expression of $g$ in the form $g=a_1b_1 \cdots a_mb_m$ where $a_1 \in A$, $a_i \in A \setminus \{\groupid\}$ for all $i \ge 2$, $b_i \in B \setminus 
\{\groupid\}$ for all $i \le m-1$, and $b_m \in B$.  The nontrivial elements $a_i,b_i$ in this factorization are the \emph{syllables} of $g$. Similarly for each word $w$ over a partitioned alphabet 
$Z_A \sqcup Z_B$, the \emph{$Z_A/Z_B$ factorization} of $w$ is the unique expression of the form $w=u_1v_1 \cdots u_nv_n$ where $u_1 \in Z_A^*$, $u_i \in Z_A^* \setminus \{\emptyword\}$ for all $i \ge 
2$, $v_i \in Z_B^* \setminus \{\emptyword\}$ for all $i \le n-1$, and $v_n \in Z_B^*$.

\begin{lemma}\label{l:free-product-prefix}
Let $A$ and $B$ be nontrivial groups with generating sets $\overline{Y_A}$ and $\overline{Y_B}$, respectively, where  $Y_A$ and $Y_B$ are finite sets, and let $G=A*B$ and $X=Y_A \sqcup Y_B \sqcup 
Y_A^{-1} \sqcup Y_B^{-1}$.  Suppose that $\phi: X^* \to G$ has the property that $\phi(x^{-1})=\phi(x)^{-1}$ for all $x \in X$.  If $w \in X^*$ and $\phi(w)$ has a reduced factorization 
$\phi(w)=a_1b_1 \cdots a_mb_m$, then for all $i$ the elements of $A*B$ with reduced factorization $a_1b_1 \cdots a_i\hat b_i$ with $\hat b_i \in \{b_i,\groupid\}$ are represented by prefixes of the 
word $w$.
\end{lemma}

\begin{proof}
We prove this for the element $g'=a_1b_1 \cdots a_ib_i$ with $b_i \neq \groupid$; the case that $\hat b_i=\groupid$ is nearly identical.  Let $Z_A=Y_A \sqcup Y_A^{-1}$ and $Z_B=Y_B \sqcup Y_B^{-1}$, 
and write the $Z_A/Z_B$ factorization of $w$ as $w=u_1v_1 \cdots u_nv_n$.

Another reduced factorization of $\phi(w)$ can be obtained from the product $(\overline{u_1})(\overline{v_1}) \cdots (\overline{u_n})(\overline{v_n})$ by finitely many applications of the following operation:  Remove a factor $\overline{u_j}$ or $\overline{v_j}$ that is the trivial element, and replace the resulting product of two contiguous factors $(\overline{u_j})(\overline{u_{k}})$ from the 
same factor group by a single factor $(\overline{u_ju_k})$ (and similarly, replace $(\overline{v_j})(\overline{v_{k}})$ by $(\overline{v_jv_k})$).  Since $\phi(w)$ only admits one reduced 
factorization, the result of applying this operation until no further reduction can occur is the product $a_1b_1 \cdots a_mb_m$.

If $m=n$, then no instance of this operation can occur, and $\overline{u_j}=a_j$ and $\overline{v_j}=b_j$ for all $j$. In this case the prefix $w'=u_1v_1 \cdots u_iv_i$ of $w$ satisfies $\phi(w')=g'$.

Otherwise we have $m > n$, and for at least one index $k$ a trivial element $\overline{u_k}$ or $\overline{v_k}$ is removed and the resulting contiguous factors from the same factor group are 
combined.  The syllables of $\phi(w)$ obtained from this process can be written as $a_j=(\overline{u_{j_1} \cdots u_{j_{\ell_j}}})$ and $b_j=(\overline{v_{j_1'} \cdots v_{j_{\ell_j'}'}})$ where 
$\ell_j,\ell_j' \ge 1$ and $j_1 < \cdots < j_{\ell_j} \le j_1' < \cdots < j_{\ell_j'}'$.  In this case the prefix $w'=u_1v_1 \cdots u_{i_{\ell_i'}'}v_{i_{\ell_i'}'}$ of $w$ satisfies $\phi(w')=g'$.
\end{proof}

\begin{proof}[Proof of Theorem~\ref{t:not-regular-free}]
By way of contradiction, assume that $G$ admits an order $\leq$ with a regular positive cone language $L$.

In light of Lemma~\ref{l:independence}, we may assume that $L\subseteq X^*$, where $X=Y_A \sqcup Y_B \sqcup Y_A^{-1} \sqcup Y_B^{-1}$, $Y_A$ and $Y_B$ are finite sets, $\overline{Y_A}$ generates $A$ 
and $\overline{Y_B}$ generates $B$, and $\phi: X^* \to G$ has the property that $\phi(x^{-1})=\phi(x)^{-1}$ for all $x \in X$. In light of Lemma~\ref{l:not-regular} it is sufficient to prove that $G 
\subseteq \overline{Pref(L)}$.

Let $g$ be an element of $G$. Since $G$ is the free product of two nontrivial groups, there exists a letter $x \in X$ such that the syllable length of $g\overline{x}$ is strictly greater than the 
syllable length of $g$. The elements $g\overline{x}g^{-1}$ and $g\overline{x}^{-1}g^{-1}$ form a pair of nontrivial, mutually inverse elements in $G$. Therefore, at least one of them is positive. 
Without loss of generality, assume that $g\overline{x}g^{-1}$ is positive.  There exists a word $w \in L$ representing $g\overline{x}g^{-1}$. By our choice of $x$, the reduced factorization of 
$g\overline{x}g^{-1}$ is the product (i.e., concatenation) of the reduced factorization of $g$, $\overline{x}$, and the reduced factorization of $g^{-1}$, and so Lemma~\ref{l:free-product-prefix} 
shows that there exists a prefix $w'$ of $w$ that represents $g$. Since $w' \in Pref(L)$ and $\overline{w'}=g$, we see that $g \in \overline{Pref(L)}$. Thus, $G \subseteq \overline{Pref(L)}$.
\end{proof}

\begin{proof}[Proof of Theorem~\ref{t:not-regular-ag}]
Let $t$ and $u$ be vertices of $\Gamma$ at distance at least 3 in $\Gamma$.  Note that this distance constraint implies that there is no vertex $v$ of $\Gamma$ that is adjacent to both $t$ and $u$.

Lemma~\ref{l:not-regular} shows that it is sufficient to prove that $Geo \subseteq Pref(Geo_+)$.

For each vertex $v$ of the graph $\Gamma$, let $Geo_v$ denote the set of geodesic words in the vertex group $G_v$ over the generators $Y_v^{\pm 1}$, and let $Z_v$ be the union of the sets $Y_{v'}$ over all vertices $v'$ adjacent to $v$ in $\Gamma$.  Define a monoid homomorphism $\pi_v:X^* \rightarrow (Y_v^{\pm 1} \cup \{\$\})^*$, where $\$$ denotes a letter not in $X$, by defining
\[
\pi_v(a) := \begin{cases}
a & \mbox{if } a \in Y_v^\pm \\
\$& \mbox{if } a \in X \setminus(Y_v^\pm \cup Z_v^\pm) \\
\emptyword & \mbox{if } a \in Z_v^\pm.\\
\end{cases}
\]
That is, the map $\pi_v$ records all of the occurrences of generators of $G_v$, as well as all of the occurrences (replaced by $\$$) of generators of vertex groups that do not commute with $G_v$ and 
hence do not allow letters from $G_v$ on either side to interact.  A word $w$ over $X$ is a geodesic for $G_\Gamma$ with respect to $X$ if and only if for all $v \in V$, $\pi_v(w) \in Geo_v(\$ 
Geo_v)^*$ (see, for example,~\cite[Proposition~3.3]{ciobanu-h:conjugacy-growth}).

Now let $w$ be any word in $Geo$. If $\pi_t(w)$ ends with a letter $a$ in $Y_t^{\pm 1}$, then we can write $w=w_1aw_2$ such that no letter of $Y_u^{\pm 1}$ lies in $w_2$, and so the word 
$\pi_u(w)=\pi_u(w_1)\$\pi_u(w_2)$ ends with a nonempty string in $\$^*$.  Similarly if $\pi_u(w)$ ends with a letter in $Y_u^{\pm 1}$ then $\pi_t(w)$ ends with $\$$.  By swapping the roles of $t$ and 
$u$ if necessary, we may assume that $\pi_t(w)$ either ends with $\$$ or is the empty word.

Let $x \in Y_t \cap Geo_t$ and $y \in Y_u \cap Geo_u$ (that is, neither $\overline{x}$ nor $\overline{y}$ is $\groupid$), and let $\tilde w \in X^*$ be the word $\tilde w=wxyx^{-1}w^{-1}$.  Then 
$\pi_v(\tilde w)=\pi_v(w)\pi_v(xyx^{-1})\pi_v(w)^{-1}$ where $\pi_v(w)^{-1}$ is the formal inverse of $\pi_v(w)$ in $(Y_v^{\pm 1} \cup \{\$\})^*$, with $\$^{-1}=\$$, and satisfies $\pi_v(w)^{-1} \in 
Geo_v(\$ Geo_v)^*$.  If $v \notin \{t,u\}$, then since $v$ cannot be adjacent to both $t$ and $u$ we have $\pi_v(xyx^{-1})=\$^i$ for some $i \in \{1,2,3\}$.  For the case that $v=t$, the word 
$\pi_t(w)$ ends with $\$$, the word $\pi_t(w)^{-1}$ begins with $\$$, and $\pi_t(xyx^{-1})=x\$x^{-1} \in Geo_t\$Geo_t$.  And in the case that $v=u$ we have $\pi_u(xyx^{-1})=\$y\$ \in \$Geo_u\$$.  
Hence for all vertices $v$ of $\Gamma$, the image $\pi_v(\tilde w)$ lies in $Geo_v(\$ Geo_v)^*$, and so $\tilde w = wxyx^{-1}w^{-1} \in Geo$.  By symmetry, the word $wxy^{-1}x^{-1}w^{-1}$ also is in 
$Geo$.

Since $wxyx^{-1}w^{-1}$ and $wxy^{-1}x^{-1}w^{-1}$ represent a pair of nontrivial, mutually inverse elements, one of them represents a positive element, which shows that $w \in Pref(Geo_+)$. Thus, 
$Geo \subseteq Pref(Geo_+)$.
\end{proof}

\subsection*{Acknowledgments} We thank the referee for their knowledgeable input.
The work of the first author was supported by the National Science Foundation under 
Grant No.~DMS-1313559.

%\bibliographystyle{alpha}
%\bibliography{../smath-new}

\end{document}